\documentclass[a4paper,11pt]{amsart}
\usepackage{longtable}
\usepackage{amsfonts}
\usepackage{amsmath}
\usepackage{amssymb}
\usepackage{latexsym}
\usepackage{amsthm}
\usepackage[all]{xy}

\newtheorem{thm}{Theorem}[section]
\newtheorem{cor}[thm]{Corollary}

\newtheorem{lem}[thm]{Lemma}

\def\af#1{\mathbb A^{#1}}
\def\aff{\op{Aff}_2(k)}
\def\al{\alpha}
\def\as#1{\renewcommand\arraystretch{#1}}

\def\be{\bigskip}
\def\cc{{\mathcal C}}
\def\comb#1#2{\as{1}\left(\!\begin{array}{c}
#1\\#2\end{array}\!\right)}
\def\combr#1#2{\as{1}\left(\!\!\left(\!\begin{array}{c}
#1\\#2
\end{array}\!\right)\!\!\right)}
\def\e{\medskip}
\def\eps{\epsilon}
\def\ff#1{\mathbb F_{#1}}
\def\fg{\op{Fix}_{\ga}}
\def\fq{\mathbb{F}_q}
\def\G{\mathbb{G}_m}
\def\g{\Gamma}
\def\ga{\gamma}
\def\gg{\Gamma_\gamma}
\def\gl{\op{GL}_2(k)}
\def\h{\mathcal{H}_g\rt}
\def\hh{\mathcal{H}_g}
\def\hp{\mathcal{H}_g^{\bullet}}

\def\kb{\overline{k}}
\def\la{\lambda}
\def\lra{\longrightarrow}
\def\md#1{\ \mbox{\rm(mod }{#1})}
\def\mdt#1{\,\mbox{\tiny\rm(mod $#1$)}}

\def\op{\operatorname}
\def\pg{\operatorname{PGL}_2(k)}
\def\pr#1{\mathbb P^{#1}}

\def\rt{^{\mbox{\scriptsize$\op{rat}$}}}
\def\sg{\sigma}
\def\tq{\,\,|\,\,}
\def\Tq{\,\,\Big{|}\,\,}

\def\xx{\mathcal{X}}
\def\yy{\mathcal{Y}}
\def\zg{\op{Fix}_{\ga}\zz}
\def\zz{\mathcal{Z}}

\title{Counting hyperelliptic curves that admit a Koblitz model}

\author{Cevahir Demirkiran and Enric Nart}

\thanks{Supported by the project MTM2006-11391 from the Spanish MEC}

\date{}
\begin{document}
\maketitle

\begin{center}
Universitat Aut\`onoma de Barcelona, Departament de Matem\`atiques \\ Edifici C, 08193 Bellaterra, Barcelona, Spain
\end{center}\e

\noindent{\bf Corresponding author: }Enric Nart, \ E-mail: {\tt nart@mat.uab.cat}, \\ Telephone Nbr. +34935811453, \ Fax Nbr. +34935812790\e

\bigskip

\begin{abstract}
Let $k=\fq$ be a finite field of odd characteristic. We find a closed formula for the number of $k$-isomorphism classes of pointed, and non-pointed, hyperelliptic curves  of genus $g$  over $k$, admitting a Koblitz model. These numbers are expressed as a polynomial in $q$ with integer coefficients (for pointed curves) and rational coefficients (for non-pointed curves). The coefficients depend on $g$ and the set of divisors of $q-1$ and $q+1$. These formulas show that the number of hyperelliptic curves of genus $g$ suitable (in principle) of cryptographic applications is asymptotically $(1-e^{-1})2q^{2g-1}$, and not $2q^{2g-1}$ as it was believed. The curves of genus $g=2$ and $g=3$ are more resistant to the attacks to the DLP; for these values of $g$ the number of curves is respectively $(91/72)q^3+O(q^2)$ and $(3641/2880)q^5+O(q^4)$.
\end{abstract}

\section*{Introduction} In a seminal paper Neal Koblitz introduced cryptosystems of El Gamal type based on the group of $k$-rational points of the Jacobian of a hyperelliptic curve over a finite field $k$ \cite{k}. In order to apply Cantor's algorithm for computing the group law of the Jacobian one works with non-singular Weierstrass equations of the type:
\begin{equation}\label{koblitz}
y^2+h(x)y=f(x),
\end{equation}
where $h(x),\,f(x)$ are polynomials in $k[x]$ of degree $\deg h(x)\le g$, $\deg f(x)=2g+1$, and the polynomial $f(x)$ is monic. The projective and smooth hyperelliptic curve $C$ obtained as the normalization of the projective closure of this affine curve has always a $k$-rational Weierstrass point at infinity. We say that the equation (\ref{koblitz}) is a {\it Koblitz model} of the curve $C$. Conversely, any hyperelliptic curve having a $k$-rational Weierstrass point admits a Koblitz model. These models have the advantage of covering simultaneously the cases of odd and even characteristic. In this paper we deal only with the odd characteristic case, and the change of variables 
$y=y-h(x)/2$ allows us to suppose $h(x)=0$.

The paper of Koblitz had an enormous impact in the cryptographic community and it was the origin of a stream of papers addressing to fundamental problems like the acceleration of the addition algorithm in the Jacobian, the computation of the number of $k$-rational points of the Jacobian, and attacks to the discrete logarithm problem. 

Some interest arose also on the problem of counting the $k$-isomorphism classes of hyperelliptic curves of a given genus, admitting a Koblitz model. For genus $2$ there is a nice review in \cite{ehm}, refering to previous work of several authors \cite{hmm,hm,cy,cj,dl}. For genus $3$ we can quote \cite{cj2,j,d}. However, all these papers count $k$-isomorphism classes of {\it pointed} hyperelliptic curves $(C,\infty)$; the distinguished point is always the Weierstrass point at infinity and two pointed curves $(C,\infty)$, $(C',\infty')$ are considered to be isomorphic if there is a $k$-isomorphism between $C$ and $C'$ sending $\infty$ to $\infty'$. P. Lockhart translated this isomorphism condition into a concrete equivalence relation between the Koblitz models \cite[Prop.1.2]{l} and in the quoted papers the authors count the number of classes of Koblitz models under this  equivalence relation. 

In this paper we use another method to find for all $g>1$ a closed formula for the isomorphism classes of pointed hyperelliptic curves of genus $g$ over finite fields of odd characteristic (Theorem \ref{hpointed}). For $g$ large the number of pointed curves is asymptotically  $2q^{2g-1}$ (Corollary \ref{ohp}).
 
Also, we solve the problem of counting the $k$-isomorphism classes of hyperelliptic curves of a given genus, admitting a Koblitz model. We give a closed formula for this number of isomorphism classes in Theorem \ref{hrat}. The dominant term of the formula is 
$$\left(1-\dfrac1{2!}+\dfrac1{3!}-\cdots-\dfrac1{(2g+2)!}\right)2q^{2g-1},$$ 
so that for $g$ large the number of curves is asymptotically $(1-e^{-1})2q^{2g-1}$ (Corollary \ref{ohrat}).
This number of isomorphism classes provides the real size of the bunch of curves suitable of cryptographic applications. For instance if $k$ is the field of $q$ elements with $q\equiv1\md3$, $q>7$, the following two genus-$2$ curves are $k$-isomorphic
$$
 y^2=x(x^2-1)(x-2)(x-3/2),\qquad  y^2=x(x^2-1)(x-1/2)(x-2/3),
$$
through the mapping $(x,y)\mapsto (1/x,y/(\sqrt{-3}x^3))$; thus, from the point of view of cryptographic applications they are identical. Nevertheless, they are not isomorphic as pointed curves. In fact, any $k$-isomorphism between the two curves preserving the point at infinity will act as $x\mapsto ax+b$ at the level of $x$-coordinates, with $a\in k^*$, $b\in k$; this map has to preserve the sets of $x$-coordinates of Weierstrass points of both  curves and it is easy to check that there is no transformation of this type sending $\{0,1,-1,2,3/2\}$ to $\{0,1,-1,1/2,2/3\}$. Thus, in the computation of isomorphism classes of pointed curves these curves count as two different curves. 

To obtain our results we use a general technique for enumerating $\pg$-orbits of rational $n$-sets of $\pr1$ that was developed in \cite{lmnx} and extended to arbitrary dimension in \cite{mn}. This technique was used in \cite{n} to obtain a formula for the total number of $k$-isomorphism classes of hyperelliptic curves. In section 1 we obtain some results on the enumeration of rational $n$-sets of algebraic varieties; the main result is Theorem \ref{quotient} where, for a given automorphism $\ga$ of $\pr1$,  we compute the number of rational $n$-sets of $\pr1$ that are fixed by $\ga$ and contain at least one rational point. In section 2 we recall some results concerning the classification of hyperelliptic curves up to $k$-isomorphism.
In section 3 we count pointed hyperelliptic curves by analyzing the action of the affine group on rational $(2g+1)$-sets of the affine line. In section 4 
we count hyperelliptic curves admitting a rational Weierstrass point by analyzing the action of the projective group on rational $(2g+2)$-sets of the projective line, containing at least one rational point.\be

\noindent{\bf Notations. }We fix once and for all a finite field $k=\fq$ of odd characteristic $p$
and an algebraic closure $\kb$ of $k$. We denote by $\sg\in\op{Gal}(\kb/k)$ the
Frobenius automorphism, $\sg(x)=x^q$. Also, $k_2$ will denote the unique quadratic extension of $k$ in $\kb$ and $\varphi$ denotes Euler's totient function.

\section{Rational $n$-sets of algebraic varieties}
Let $V$ be an algebraic variety defined over $k$. A {\it rational $n$-set of $V$} is by definition a $k$-rational point of the variety $\comb{V}{n}$ of $n$-sets of $V$. Thus,
a rational $n$-set $S\in\comb{V}{n}(k)$ is just an unordered family
$S=\{t_1,\dots,t_n\}$ of $n$ different points of $V(\kb)$, which is globally invariant under the Galois action: $S=S^{\sigma}$.

For any subset $Z\subseteq V(k)$ of $k$-rational points of $V$ we denote
$$
\comb{V}{n}_{\!Z}:=\left\{S\in \comb{V}{n}(k)\Tq S\cap Z=\emptyset\right\}.
$$$$
\comb{V}{n}^{\!Z}:=\left\{S\in \comb{V}{n}(k)\Tq S\cap Z\ne\emptyset\right\}.
$$
For instance, for $Z=V(k)$ we obtain in the last case the set of rational $n$-sets of $V$ containing at least one $k$-rational point. In the cases $Z=V(k)$ and $Z=\{P\}$ we use a special notation 
$$
\comb Vn\rt:=\comb Vn^{V(k)},\qquad \comb Vn^P:=\comb Vn^{\{P\}}.
$$

For any pair $r,\,n$ of non-negative integers we denote:
$$
a_V(r,n):=\left|\comb{V}n_{\!Z}\right|,\quad b_V(r,n):=\left|\comb Vn^{\!Z}\right|,
$$ where $Z$ is any subset of $V(k)$ with $|Z|=r$. Also, we introduce a particular notation
for the extreme cases:
$$a_V(n):=a_V(0,n)=\left|\comb{V}n(k)\right|, \qquad b_V(n):=b_V(|V(k)|,n)=\left|\comb Vn\rt\right|.
$$Hence, $a_V(n)$ counts the total number of rational $n$-sets of $V$ whereas $b_V(n)$ counts the number of rational $n$-sets of $V$ that contain at least one rational point.

Since $\comb{V}n^{\!Z}$ and $\comb Vn_{\!Z}$ are complementary subsets of $\comb Vn(k)$ we have, for all $r,\,n\ge0$:
\begin{equation}\label{compl}
a_V(r,n)+b_V(r,n)=a_V(n).
\end{equation}

It is easy to compute $a_V(r,n)$ in terms of the function $a_V$:

\begin{lem}
For any algebraic variety $V$ defined over $k$:
$$
a_V(r,n)=\sum_{i=0}^n(-1)^i\combr ria_V(n-i).
$$
\end{lem}

\begin{proof}
We proceed by induction on $r$. Let $Z=\{t\}$ for some $t\in V(k)$. Distributing the rational $n$-sets of $V$ into two families according to the fact that they contain $t$ or not we see that
$$
a_V(n)=a_V(1,n)+a_V(1,n-1).
$$By Moebius inversion we get
\begin{equation}\label{m1}
a_V(1,n)=\sum_{i=0}^n(-1)^ia_V(n-i),
\end{equation}
and the statement of the lemma is proven for $r=1$.

Suppose that the claim has been checked for all varieties $V$ and all subsets $Z\subseteq V(k)$ with $|Z|\le r-1$:
$$
a_V(r-1,n)=\sum_{i=0}^n(-1)^i\combr {r-1}ia_V(n-i).
$$
By (\ref{m1}) we have
$$
a_{V}(r,n)=\sum_{i=0}^n(-1)^ia_V(r-1,n-i),
$$
and using the two formulas we get
\begin{multline*}
a_V(r,n)=\sum_{i=0}^n(-1)^i\left(\combr{r-1}0+\cdots+\combr{r-1}i\right)a_V(n-i)=\\=
\sum_{i=0}^n(-1)^i\combr ria_V(n-i).
\end{multline*}
\end{proof}

By (\ref{compl}) we get immediately a computation of  $b_V(r,n)$:

\begin{cor}\label{bmn}
For any algebraic variety $V$ defined over $k$:
$$
b_V(r,n)=\sum_{i=1}^n(-1)^{i+1}\combr ria_V(n-i).
$$
\end{cor}

We prove now a result that will be crucial in the enumeration of $\pg$-orbits of
rational $n$-sets of $\pr1$. The projective action of $\pg$ on $\pr1(\kb)$ induces a natural action of $\pg$ on the set of rational $n$-sets of $\pr1$.  

For any $\ga\in\pg$ we denote by $\fg$ the set of fixed points of $\ga$ in $\pr1(\kb)$. More generally, if $X$ is a set admitting an action of $\ga$ we denote by $\fg X$ the subset of
fixed points of $\ga$ in $X$.   
 
\begin{thm}\label{quotient}
Let $\ga$ be an element of $\pg$, $\ga\ne 1$, and let $m$ be the order of $\ga$. Let $V$ be the open subvariety $\pr1\setminus \fg$ of $\pr1$. Then, for any positive integer $n$
$$
\left|\fg \comb Vn(k)\right|=a_V\left(n/m\right),
$$ 
$$
\left|\fg \comb Vn\rt\right|=b_V\left(|V(k)|/m,n/m\right),
$$ 
with the convention that $a_V(x)=0=b_V(r,x)$ if $x$ is not a positive integer. 
\end{thm}

\begin{proof}
Let $\pr1/\ga$ be the quotient variety of $\pr1$ under the action of the cyclic group generated by $\ga$. The curve $\pr1/\ga$ is $k$-isomorphic to $\pr1$ because it is normal and birrationally equivalent to $\pr1$ (by L\"uroth's theorem). Also, the Zariski closed set $(\pr1/\ga)\setminus(V/\ga)$ is isomorphic to $\fg$ as a Galois set; therefore, $V/\ga$ is $k$-isomorphic to $V$ too and
$a_{V/\ga}(n)=a_{V}(n)$, $b_{V/\ga}(r,n)=b_{V}(r,n)$, for all $r$, $n$. 

Consider the canonical projection
$$
\pi\colon V\lra V/\ga.
$$ For any $t\in V(\kb)$ the $\ga$-orbit $O_{\ga}(t)=\{t,\ga(t),\dots,\ga^{m-1}(t)\}$ has cardinality $m$ (and not a proper divisor of $m$) \cite[Lem.2.3]{lmnx}. Thus, if an $n$-set of $V$ is $\ga$-invariant then necessarily $n$ is a multiple of $m$. On the other hand, the mapping $\pi$ establishes a 1-1 correspondence between $\ga$-invariant $n$-rational sets of $V$ and rational $n/m$-sets of $V/\ga$. In particular, 
$$
\left|\fg \comb V{n}(k)\right|=a_{V/\ga}\left(n/m\right)=a_V\left(n/m\right).
$$

The $\ga$-orbits 
$O:=O_{\ga}(t)$ such that $O=O^{\sg}$ are in 1-1 correspondence with the set of $k$-rational points of $V/\ga$. Exactly $|V(k)|/m$ of these orbits have the property that $O$ contains a $k$-rational point (or equivalently all points of $O$ are $k$-rational); thus, the $\ga$-invariant rational $m$-sets of $V$ that contain at least one $k$-rational point are in 1-1 correspondence with certain subset $Z\subseteq (V/\ga)(k)$ of cardinality $|V(k)|/m$. Therefore, $\pi$ determines a 1-1 correspondence between $\ga$-invariant $n$-rational sets of $V$ containing at least one $k$-rational point, and rational $n/m$-sets of $V/\ga$ containing at least one point of $Z$. Hence,
$$
\left|\fg \comb Vn\rt\right|=b_{V/\ga}\left(\dfrac{|V(k)|}m,\frac nm\right)=b_V\left(\dfrac{|V(k)|}m,\frac nm\right).
$$
\end{proof}

In sections 3 and 4 we shall express the number of isomorphism classes of pointed and non-pointed hyperelliptic curves admitting a rational Weierstrass point, in terms of $a_V(n)$ and $b_V(r,n)$ for the varieties 
$$V=\pr1,\ \af1,\ \G,\ \pr1_0,$$ where $\pr1_0$ is the subvariety $\pr1\setminus\{t,t^{\sg}\}$, being $t$ any point in $\pr1(k_2)\setminus\pr1(k)$. 
Formulas for $a_V(n)$ for these four varieties were found in \cite[Lem. 2.1]{lmnx} and the value of $b_V(r,n)$ is deduced from Corollary \ref{bmn}.
Actually, we shall use certain normalizations of these numbers. The following lemma collects all the formulas we need.

\begin{lem}\label{atilla}
For positive integers $n$, $m$ we have
$$
a_{\pr{1}}(n)=\left\{\begin{array}{ll}
q^n-q^{n-2},&\quad \mbox{ if } \,n\ge3,\\
q^2,&\quad \mbox{ if } \,n=2,\\
q+1,&\quad \mbox{ if } \,n=1.
\end{array}\right.
$$
$$
A_1(n):=\dfrac{a_{\af1}(n)}q=\left\{\begin{array}{ll}
q^{n-1}-q^{n-2},&\mbox{ if }n\ge2,\\
1,&\mbox{ if }n=1.\end{array}\right.
$$
$$
A_2(n):=\dfrac{a_{\G}(n)}{q-1}=\dfrac{q^n-(-1)^n}{q+1}.
$$
$$
A_0(n):=\dfrac{a_{\pr1_0}(n)}{q+1}=
\dfrac{q^{n+1}-q^n-(-1)^{\lceil n/2\rceil}q+(-1)^{\lceil(n-1)/2\rceil}}{q^2+1}.
$$
\begin{multline*}
B(n):=\dfrac{b_{\pr1}(n)}{q(q-1)(q+1)}=\sum_{i=1}^{n-3}(-1)^{i+1}\combr{q+1}iq^{n-3-i}-\\-(-1)^n\dfrac{n-1}{n(q+1)}\combr{q+1}{n-2}, \ \forall n>3.
\end{multline*}
\begin{multline*}
B_0(m,n):=\dfrac1{q+1}b_{\pr1_0}\left(\frac{q+1}m,n\right)=\\=\sum_{i=1}^{n-1}(-1)^{i+1}\combr{(q+1)/m}iA_0(n-i)-\dfrac{(-1)^n}{q+1}\combr{(q+1)/m}n.
\end{multline*}
\begin{multline*}
B_1(n):=\dfrac1qb_{\af1}\left(\frac qp,n\right)=\\=\sum_{i=1}^{n-1}(-1)^{i+1}\combr{q/p}iA_1(n-i)-\dfrac{(-1)^n}q\combr{q/p}n.
\end{multline*}
\begin{multline*}
B_2(m,n):=\dfrac1{q-1}b_{\G}\left(\frac{q-1}m,n\right)=\\=\sum_{i=1}^{n-1}(-1)^{i+1}\combr{(q-1)/m}iA_2(n-i)-\dfrac{(-1)^n}{q-1}\combr{(q-1)/m}n.
\end{multline*}
\end{lem}

\section{Classification of hyperelliptic curves up to $k$-isomorphism}
In this section we recall the connection between rational sets of $\pr1$ and hyperelliptic curves over $k$. For generalities on hyperelliptic curves we address the reader to \cite[Sec.1]{n}.

From now on we assume that $n=2g+2$, where $g$ is a positive integer, $g>1$. To each rational $n$-set $S$ of $\pr1$ we can attach the monic
separable polynomial $f_S(x)\in k[x]$ of degree $n$ or $n-1$ given
by:
$$
f_S(x):=\prod_{t\in S,\ t\ne\infty}(x-t).
$$

 To every $\la\in k^*$, $S\in
\comb{\pr1}{n}(k)$, we can attach the hyperelliptic
curve $C_{\la,S}$ determined by the Weierstrass equation $y^2=\la
f_S(x)$.

For any $\mu\in k^*$ the morphism $(x,y)\mapsto(x,\mu y)$ sets a
$k$-isomorphism between  $C_{\la,S}$ and $C_{\la\mu^2,S}$. Thus, if we let the pairs $(\la,S)$ run on the set 
$$
(\la,S)\in\xx_n:=\left(k^*/(k^*)^2\right)\times \comb{\pr1}n(k),
$$
the curves $C_{\la,S}$ contain representatives of all
$k$-isomorphism classes of hyperelliptic curves of genus $g$.

The natural action of $\pg$ on $n$-sets of $\pr1$ determines a
natural action of $\pg$ on the set of hyperelliptic curves defined
over $k$. In order to recall this action we introduce multipliers $J(\ga,S)\in k^*$ that depend in principle on the choice of a representative in $\gl$ of $\ga\in\pg$. Consider a matrix
$$
\ga=\begin{pmatrix}a&b\\c&d\end{pmatrix}\in\gl.
$$
For any $t\in\pr1(\kb)$ we can define a local multiplier
$j(\ga,t)\in \kb^*$ by
$$
j(\ga,t):=\left\{\begin{array}{ll}\det(\ga)(ct+d)^{-1}&\mbox{ if
}t\ne\infty,\,t\ne -d/c\\c&\mbox{ if }t= -d/c,\,c\ne0\\d&\mbox{ if
}t=\infty,\,c=0\\-\det(\ga)c^{-1}&\mbox{ if
}t=\infty,\,c\ne0\end{array}\right.
$$
For any rational $n$-set $S$ of $\pr1$ we define a global
multiplier $$J(\ga,S):=\prod_{t\in S}j(\ga,t)\in k^*.$$ 

There is a well-defined  action of $\pg$ on the set $\xx_n$:
$$
\ga(\la,S):=(\la J(\ga,S),\ga(S)),
$$which is independent of the choice of a representative of $\ga\in\pg$ in $\gl$. The map $(\la,S)\mapsto
C_{\la,S}$ induces a 1-1 correspondence
$$
\pg\backslash\xx_n \lra \hh,
$$
where $\hh$ is the set of $k$-isomorphism classes of hyperelliptic curves over $k$ of genus $g$ \cite[Thm.2.4]{n}. 

Let us adapt this result to the situation we are dealing with in this paper. Recall that a {\it pointed hyperelliptic curve} is for us a pair $(C,P)$ where $C$ is a hyperelliptic curve over $k$ and $P$ is a rational Weierstrass point of $C$. We say that two pointed curves $(C,P)$, $(C',P')$ are $k$-isomorphic if there is a $k$-isomorphism between $C$ and $C'$ sending $P$ to $P'$. Denote by $\hp$ the set of $k$-isomorphism classes of pointed hyperelliptic curves of genus $g$. On the other hand, denote by $\h$ the set of $k$-isomorphism classes of hyperelliptic curves of genus $g$ admitting at least one rational Weierstrass point.

Consider the sets 
$$
\yy_n:=\left(k^*/(k^*)^2\right)\times \comb{\pr1}n^{\infty},\qquad
\zz_n:=\left(k^*/(k^*)^2\right)\times \comb{\pr1}n\rt.
$$
These subsets of $\xx_n$ are stable under the action of $\pg$. Moreover, the set $\yy_n$ is stable under the action of the affine subgroup, which is the stabilizer of the point $\infty\in\pr1(k)$: 
$$
\aff:=\left\{\begin{pmatrix}a&b\\0&1\end{pmatrix}\Tq (a,b)\in k^*\times k\right\}\subseteq \pg.
$$The following result is an immediate consequence
of \cite[Thm.2.4]{n}. 

\begin{thm}\label{red}
The map $(\la,S)\mapsto
C_{\la,S}$ induces 1-1 correspondences
$$
\aff\backslash\yy_n \lra \hp,\qquad \pg\backslash\zz_n \lra \h.
$$
\end{thm}

After this result the aim of the paper is to find closed formulas for the cardinalities of the two sets $\aff\backslash\yy_n$, $\pg\backslash\zz_n$.
To this end we need the computation of the class of $J(\ga,S)$ modulo squares given in \cite[Thm.3.4]{n}, which we recall in Theorem \ref{eps} below.

Denote by $\eps$ the map
$$\eps\colon\pg\times \comb{\pr1}n (k)\stackrel{J}{\lra} k^*/(k^*)^2 \lra \{\pm1\},
$$
where the last map is the unique non-trivial group homomorphism
between these two groups of order two. We want to compute $\eps(\ga,S)$ for $\ga$ running on a system of representatives of conjugacy classes of $\pg$, and $S$ a rational $n$-set of $\pr1$ fixed by $\ga$: $\ga(S)=S$. 

Let us recall how these representatives can be chosen, the possible values of the order $m$ of $\ga$ in each conjugacy class, the number of representatives of a given order and the cardinality of the centralizers
$$
\gg:=\{\rho\in\pg\tq \rho^{-1}\ga\rho=\ga\}.
$$The following result is extracted from \cite[Prop.2.3, Lem.2.4]{ln}.

\begin{lem}\label{rmk}
There are $q+2$ conjugacy classes in $\pg$, which we distribute in four types:\e

A. The identity, $\ga(t)=t$, has order $m=1$ and $|\gg|=|\pg|=q(q-1)(q+1)$.\e

B. The translation $\ga_0(t)=t+1$. It has $\op{Fix}_{\ga_0}=\{\infty\}$, order $m=p$ and $|\g_{\ga_0}|=q$. \e

C. The homothetic automorphisms (conjugate to $t\mapsto \la t$, for some $\la
\in k^*$, $\la\ne1$). They have two fixed points, lying in $\pr1(k)$, and order $m=\op{ord}_{k^*}(\la)$, which is a divisor of $q-1$. 

There are $(q-1)/2$ homothetic conjugacy classes. For any divisor $m>1$ of $q-1$, if $\cc_m$ is a system of representatives of the homothetic conjugacy classes of order $m$, we have
$$
\sum_{\ga\in\cc_m}|\gg|^{-1}=\dfrac {\varphi(m)}{2(q-1)}.
$$

D.  The potentially homothetic automorphisms; i.e. those $\ga$ conjugate to the class in $\pg$ of a matrix
$\begin{pmatrix}0&1\\c&d\end{pmatrix}\in\gl$ with eigenvalues
$\al,\,\al^{\sg}$ in $k_2\setminus k$. They have two fixed points, which are quadratic
conjugate points in $\pr1(k_2)$; the order is the least
positive integer $m$ such that $\al^m\in k$, and it is a divisor of
$q+1$.

There are $(q+1)/2$ potentially homothetic conjugacy classes. For any divisor $m>1$ of $q+1$, if $\cc_m$ is a system of representatives of the potentially homothetic conjugacy classes of order $m$, we have
$$
\sum_{\ga\in\cc_m}|\gg|^{-1}=\dfrac {\varphi(m)}{2(q+1)}.
$$
\end{lem}

\begin{thm}\label{eps}
Let $n$ be an even positive integer, $S$ a rational $n$-set of $\pr1$, and $\ga\in\pg$ an automorphism of $\pr1$ of order $m$ such that $\ga(S)=S$. Then, 
$$
\eps(\ga,S)=\left\{ \begin{array}{ll} 
1,&\mbox{if $\ga=1$ or }\ga=\ga_0\\
(-1)^{(q-1)/m},&\mbox{if $\ga$ homothetic and } \infty\in S\\
1,&\mbox{if $\ga$ homothetic and } \infty\not\in S\\
(-1)^{(q+1)/m}(-1)^{(n-2)/m},&\mbox{if
$\ga$ pot. homothetic and }\fg\subseteq S
\\(-1)^{n/m},&\mbox{if
$\ga$ pot. homothetic and }\fg\not\subseteq S
\end{array}\right.
$$
\end{thm}

\section{Counting pointed hyperelliptic curves}
Let $\g$ be a finite group acting on a finite set $X$. The number of orbits of this action can be counted as the average number of fixed points:
\begin{equation}\label{cf} |\g\backslash X|=\frac
1{\vert\g\vert}\sum_{\ga\in\g}\vert \fg X\vert=
\sum_{\ga\in\cc}\frac{\vert \fg X\vert}{\vert \g_{\ga}\vert},
\end{equation} where  $\cc$ is a set of representatives of conjugacy classes of elements of
$\g$ and
 $$
 \fg X:=\{x\in X \tq \ga(x)=x\}, \quad
 \g_{\ga}:=\{\rho\in \g \tq \rho\ga\rho^{-1}=\ga\}.
 $$

In this section we apply this formula to compute the number $\op{hyp}^{\bullet}(g)$ of orbits of the set $$X=\yy:=\yy_{2g+2}=\left(k^*/(k^*)^2\right)\times \comb{\pr1}{2g+2}^{\infty}$$ under the action of the affine group $\g:=\aff$. By Theorem \ref{red} this is the number of $k$-isomorphism classes of pointed hyperelliptic curves of genus $g$: $\op{hyp}^{\bullet}(g)=|\hp|$. 

The following lemma exhibits a system of representatives of conjugacy classes of the affine group:

\begin{lem}\label{rmkaff}
There are $q$ conjugacy classes in $\aff$, represented by the following elements, which we distribute in three types:\e

A. The identity, $\ga(t)=t$, has order $m=1$ and $|\gg|=|\aff|=q(q-1)$.\e

B. The translation $\ga_0(t)=t+1$. It has $\op{Fix}_{\ga_0}=\{\infty\}$, order $m=p$ and $|\g_{\ga_0}|=q$. \e

C. The homotheties $\ga(t)= \la t$, $\la\in k^*$, $\la\ne1$. They have $\fg=\{\infty,\,0\}$, order $m=\op{ord}_{k^*}(\la)$, which is a divisor of
$q-1$, and $|\gg|=q-1$. 

In particular, for any divisor $m>1$ of $q-1$ there are $\varphi(m)$ conjugacy classes in $\aff$ of order $m$.
\end{lem}

For any $\ga\in\aff$, a pair $(\la,S)\in\yy$ is fixed by $\ga$ if and only if $\ga(S)=S$ and $\eps(\ga,S)=1$. Thus,
$$
\left|\fg\yy\right|=2\left|\left\{S\in \fg\comb{\pr1}{2g+2}^{\infty}\Tq \eps(\ga,S)=1\right\}\right|.
$$
By Theorem \ref{eps}, $\left|\fg\yy\right|=0$ if $\ga$ is an homothety of order $m$ with $(q-1)/m$ odd, and 
$$
\left|\fg\yy\right|=2\left|\fg\comb{\pr1}{2g+2}^{\infty}\right|=2\left|\fg\comb{\af1}{2g+1}(k)\right|,
$$ otherwise. We can apply now Theorem \ref{quotient} to compute the number of rational $(2g+1)$-sets of $\af1$ which are $\ga$-invariant. If $\ga$ is an homothety, in order to be able to apply Theorem \ref{quotient} we split these
rational $(2g+1)$-sets into two disjoint groups according to the fact that they contain $0$ or not; we get in this case
$$
\left|\fg\comb{\af1}{2g+1}(k)\right|=\left|\fg\comb{\G}{2g}(k)\right|+\left|\fg\comb{\G}{2g+1}(k)\right|,
$$  
and we obtain
$$
\left|\fg\comb{\af1}{2g+1}(k)\right|=\left\{
\begin{array}{ll}
a_{\af1}(2g+1),&\quad\mbox{ if }\ga=1,\\
a_{\af1}((2g+1)/p),&\quad\mbox{ if }\ga=\ga_0,\\
a_{\G}(2g/m)+a_{\G}((2g+1)/m),&\quad\mbox{ otherwise},
\end{array}
\right.
$$where $m$ is the order of $\ga$. 

The computation of $\op{hyp}^{\bullet}(g)$ given by (\ref{cf}) can be splitted into the sum of three terms $h_A+h_B+h_C$, each term taking care of the contribution of all conjugacy classes in a concrete type, as described in Lemma \ref{rmkaff}. Since  the value of $\left|\fg\yy\right|$ depends only on $m$, for the computation of $h_C$ we can group together all $\ga$ with the same order and we obtain

\begin{multline*}
\op{hyp}^{\bullet}(g)=
\dfrac{2a_{\af1}(2g+1)}{q(q-1)}+
\frac 2qa_{\af1}\left(\frac{2g+1}p\right)+\\+
\dfrac2{q-1}\sum_{1<m|(q-1)/2}\varphi(m)\left(a_{\G}\left(\frac{2g}m\right)+a_{\G}\left(\frac{2g+1}m\right)\right)=\\=
2q^{2g-1}+2A_1\left(\frac{2g+1}p\right)+\\+
\sum_{1<m|(q-1)/2}2\varphi(m)\left(A_2\left(\frac{2g}m\right)+A_2\left(\frac{2g+1}m\right)\right).\end{multline*}

By using the explicit formulas for $A_1(n)$, $A_2(n)$ given in Lemma \ref{atilla} we obtain a closed formula for $\op{hyp}^{\bullet}(g)$ as a polynomial in $q$ with integer coefficients that depend on $g$ and the set of divisors of $q-1$. This is more clearly seen if we rewrite our formula for $\op{hyp}^{\bullet}(g)$ in a way that is more suitable for an effective computation when $g$ is given and we want to deal with a generic value of $q$.

\begin{thm}\label{hpointed}
The number of $k$-isomorphism classes of pointed hyperelliptic curves of genus $g$ is:
\begin{multline*}
\op{hyp}^{\bullet}(g)=2q^{2g-1}+2A_1\left(\dfrac {2g+1}p\right)+ \\
+\sum_{1<m|2g+1}2\varphi(m)\left[A_2\left(\frac{2g+1}m\right)\right]_{m|q-1}+
\sum_{1<m|2g}2\varphi(m)\left[A_2\left(\frac{2g}m\right)\right]_{2m|q-1}.
\end{multline*}
By convention, $A_1(x)=0$ if $x$ is not a positive integer and the terms $[x]_{\mbox{\tiny condition}}$ are considered only if the ``condition" is satisfied. 
\end{thm}

We display in Table 1 the value of $\op{hyp}^{\bullet}(g)$ for $2\le g\le 7$. 

\begin{center}
\begin{table}\as{1.4}
\caption{\small Number of pointed hyperelliptic curves of genus $g$ up to $k$-isomorphism}
\begin{tabular}{|c|l|}\hline
\mbox{\small$g$}&\qquad\qquad\qquad\qquad\qquad\ \mbox{\small$\op{hyp}^{\bullet}(g)=|\hp|$}\\\hline
\mbox{\small$2$}&\mbox{\small$2q^3+2[q-1]_{4|q-1}+[4]_{8|q-1}+[8]_{5|q-1}+[2]_{p=5}$}\\\hline
\mbox{\small$3$}&\mbox{\small$2q^5+2[q^2-q+1]_{4|q-1}+4[q-1]_{3|q-1}+[12]_{7|q-1}+[4]_{12|q-1}+[2]_{p=7}$}\\\hline
\mbox{\small$4$}&\negmedspace\negmedspace\mbox{\small$\begin{array}{l}
2q^7+2[q^3-q^2+q-1]_{4|q-1}+4[q^2-q+1]_{3|q-1}+4[q-1]_{8|q-1}+\\
\qquad\quad+[12]_{9|q-1}+[8]_{16|q-1}+2[q^2-q]_{p=3}\end{array}$}\\\hline
\mbox{\small$5$}&\negmedspace\negmedspace\mbox{\small$\begin{array}{l}
2q^9+2[q^4-q^3+q^2-q+1]_{4|q-1}+8[q-1]_{5|q-1}+[20]_{11|q-1}+[8]_{20|q-1}+\\
\qquad\quad+[2]_{p=11}\end{array}$}\\
\hline\mbox{\small$6$}&\negmedspace\negmedspace\mbox{\small$\begin{array}{l}
2q^{11}+2[q^5-q^4+q^3-q^2+q-1]_{4|q-1}+4[q^3-q^2+q-1]_{3|q-1}+\\
\qquad\quad+4[q^2-q+1]_{8|q-1}+4[q-1]_{12|q-1}+[24]_{13|q-1}+[8]_{24|q-1}+[2]_{p=13}\end{array}$}\\\hline
\mbox{\small$7$}&\negmedspace\negmedspace\mbox{\small$\begin{array}{l}
2q^{13}+2[q^6-q^5+q^4-q^3+q^2-q+1]_{4|q-1}+4[q^4-q^3+q^2-q+1]_{3|q-1}+\\
\qquad\quad+8[q^2-q+1]_{5|q-1}+12[q-1]_{7|q-1}+[16]_{15|q-1}+[12]_{28|q-1}+\\
\qquad\quad\qquad\qquad+2[q^4-q^3]_{p=3}+2[q^2-q]_{p=5}\end{array}$}\\
\hline
\end{tabular}
\end{table}
\end{center}

\begin{cor}\label{ohp}
The dominant terms of $\op{hyp}^{\bullet}(g)$ are 
$$\op{hyp}^{\bullet}(g)=2q^{2g-1}+O(q^{g-1}).$$
\end{cor}

\begin{proof}
Apart from the generic term $2q^{2g-1}$, the highest power of $q$ arising from the other terms is the degree of $\left[A_2(g)\right]_{4|q-1}$, corresponding to the divisor $m=2$ of $2g$.
\end{proof}

\section{Counting hyperelliptic curves with a rational Weierstrass point}
In this section we apply the formula (\ref{cf}) to compute the number $\op{hyp}\rt(g)$ of orbits of the set $$X=\zz:=\zz_{2g+2}=\left(k^*/(k^*)^2\right)\times \comb{\pr1}{2g+2}\rt$$ under the action of the projective group $\g:=\pg$. By Theorem \ref{red} this is the number of $k$-isomorphism classes of hyperelliptic curves of genus $g$ having at least one rational Weierstrass point: $\op{hyp}\rt(g)=|\h|$. 

We divide the conjugacy classes of $\pg$ into four types A, B, C, D, as indicated in Lemma \ref{rmk}. The computation of $\op{hyp}\rt(g)$ given by (\ref{cf}) can be splitted into the sum of four terms $\op{hyp}\rt(g)=h_A+h_B+h_C+h_D$, taking care of the contribution of all conjugacy classes of each concrete type. 

For any $\ga\in\pg$, a pair $(\la,S)\in\zz$ is fixed by $\ga$ if and only if $\ga(S)=S$ and $\eps(\ga,S)=1$. Thus,
$$
\left|\zg\right|=2\left|\left\{S\in \fg\comb{\pr1}{2g+2}\rt\Tq \eps(\ga,S)=1\right\}\right|.
$$

For $\ga=1$ and $\ga=\ga_0$ Theorem \ref{eps} shows that $\eps(\ga,S)=1$ 
for all $\ga$-invariant rational $(2g+2)$-sets of $\pr1$. Thus, 
$$\left|\zg\right|=2\left|\fg\comb{\pr1}{2g+2}\rt\right|.$$ For $\ga=1$ we get directly
$$
h_A=\dfrac{|\zz|}{q(q-1)(q+1)}=\dfrac{2b_{\pr1}(2g+2)}{q(q-1)(q+1)}=2B(2g+2).
$$
We split the $\ga_0$-invariant rational $(2g+2)$-sets of $\pr1$ that contain at least one rational point into two families: those containing $\infty$ and those not containing $\infty$. We have,
$$
\left|\op{Fix}_{\ga_0}\comb{\pr1}{2g+2}\rt\right|=
\left|\op{Fix}_{\ga_0}\comb{\af1}{2g+1}(k)\right|+
\left|\op{Fix}_{\ga_0}\comb{\af1}{2g+2}\rt\right|.
$$
Theorem \ref{quotient} can be applied to compute the cardinality of both families, and we get
\begin{multline*}
h_B=\dfrac{|\op{Fix}_{\ga_0}\zz|}q=\dfrac 2q \left(a_{\af1}\left(\frac{2g+1}p\right)+b_{\af1}\left(\dfrac qp,\frac {2g+2}p\right)\right)=\\=2A_1\left(\frac{2g+1}p\right)
+2B_1\left(\frac {2g+2}p\right).
\end{multline*}

For $\ga$ of type C or D we shall see below that the value of $|\zg|$ depends only on the order $m$ of $\ga$; hence, in the computation of $h_C$ and $h_D$ we can group together all $\ga$ with the same order and Lemma \ref{rmk} shows that we can express the partial sums $h_C$ and $h_D$ as:
$$
h_C=\sum_{1<m|(q-1)}\dfrac{\varphi(m)|\zg|}{2(q-1)},
\quad h_D=\sum_{1<m|(q+1)}\dfrac{\varphi(m)|\zg|}{2(q+1)},
$$
where for each $m$ we choose an arbitrary $\ga$ of order $m$ of type C or D.

For $\ga$ an homothety of order $m$, we split the $\ga$-invariant rational $(2g+2)$-sets of $\pr1$ that contain at least one rational point into three families: those containing both fixed points of $\ga$ ($0$ and $\infty$), those containing exactly one fixed point of $\ga$, and those containing no fixed points of $\ga$. Since any $\ga$-invariant $n$-set of $\G$ has necessarily $n$ multiple of $m$, Theorem \ref{eps} shows that:
$$
|\zg|=2\left|\fg\comb {\G}{2g+2}\rt\right|,\quad\mbox{ if }\,\frac{q-1}m\,\mbox{ odd},
$$
\begin{multline*}
|\zg|=2\left|\fg\comb {\G}{2g}(k)\right|+4\left|\fg\comb {\G}{2g+1}(k)\right|+\\+2\left|\fg\comb {\G}{2g+2}\rt\right|,\quad\mbox{ if }\,\frac{q-1}m\,\mbox{ even}.
\end{multline*}

Theorem \ref{quotient} can be applied to compute the cardinality of each family, and
we get
\begin{multline*}
h_C=\sum_{1<m|q-1}\dfrac{\varphi(m)}{q-1}\left(\left[a_{\G}\left(\frac{2g}m\right)\right]_{2m|q-1}+2a_{\G}\left(\frac{2g+1}m\right)+\right.\\
\left.+b_{\G}\left(\dfrac{q-1}m,\frac{2g+2}m\right)\right)=\sum_{1<m|q-1}\varphi(m)\left(\left[A_2\left(\frac{2g}m\right)\right]_{2m|q-1}+\right.\\\left.+2A_2\left(\frac{2g+1}m\right)+B_2\left(\dfrac{q-1}m,\frac{2g+2}m\right)
\right).
\end{multline*}
Finally, $h_D$ can be computed by using completely analogous arguments:
\begin{multline*}
h_D=\sum_{1<m|q+1}\dfrac{\varphi(m)}{q+1}\left(\left[b_{\pr1_0}\left(\dfrac{q+1}m,\dfrac {2g}m\right)\right]_{\frac {2g}m \equiv \frac {q+1}m\mdt2}+\right.\\\left.
+\left[b_{\pr1_0}\left(\dfrac{q+1}m,\dfrac {2g+2}m\right)\right]_{m|g+1}\right)=\\=\sum_{1<m|q+1}\varphi(m)\left(\left[B_0\left(m,\dfrac {2g}m\right)\right]_{\frac {2g}m \equiv \frac {q+1}m\mdt2}+\right.\\\left.+\left[B_0\left(m,\dfrac {2g+2}m\right)\right]_{m|g+1}
\right).
\end{multline*}

By using the explicit formulas for $A_i(n)$ and $B_i(m,n)$ given in Lemma \ref{atilla} we obtain a closed formula for $\op{hyp}\rt(g)$ as a polynomial in $q$ with rational coefficients that depend on the set of divisors of $q-1$ and $q+1$. As in the previous section we rewrite our computation of $\op{hyp}\rt(g)$ in a way that is more suitable for an effective computation when $g$ is given and we want to deal with a generic value of $q$.

\begin{thm}\label{hrat}
The number of $k$-isomorphism classes of hyperelliptic curves of genus $g$ having at least one rational Weierstrass point is:
\begin{multline*}
\op{hyp}\rt(g)=2B(2g+2)+2A_1\left(\frac{2g+1}p\right)+2B_1\left(\frac{2g+2}p\right)+\\+
\sum_{1<m|2g}\varphi(m)\left(\left[B_0\left(m,\dfrac {2g}m\right)\right]_{m|q+1,\,\frac {2g}m \equiv \frac {q+1}m\mdt2}+\left[A_2\left(\frac{2g}m\right)\right]_{2m|q-1}\right)+\\+
\sum_{1<m|2g+1}2\varphi(m)\left[A_2\left(\frac{2g+1}m\right)\right]_{m|q-1}+\\+
\sum_{1<m|2g+2}\varphi(m)\left(\left[B_0\left(m,\dfrac {2g+2}m\right)\right]_{m|q+1,\,m|g+1}+\right.\\\left.+\left[B_2\left(m,\frac{2g+2}m\right)\right]_{m|q-1}\right).\end{multline*}
By convention, $A_1(x)=0=B_1(x)$ if $x$ is not a positive integer and the terms $[x]_{\mbox{\tiny condition}}$ are considered only if the ``condition" is satisfied. 
\end{thm}

We display in Table 2 the value of $\op{hyp}\rt(g)$ for $2\le g\le 5$.  
\begin{center}
\begin{table}\as{1.4}
\caption{\small Number of hyperelliptic curves of genus $g$ admitting a Koblitz model, up to $k$-isomorphism}
\begin{tabular}{|c|l|}\hline
\mbox{\small$g$}&\qquad\qquad\qquad\qquad\qquad\ \mbox{\small$\op{hyp}\rt(g)=|\h|$}\\\hline
\mbox{\small$2$}&\negmedspace\negmedspace\mbox{\small$\begin{array}{l}\frac{91}{72}q^3+\frac{37}{48}q^2-\frac12q+\frac{11}{16}+[q-1]_{4|q-1}+\frac18[3q+1]_{4|q+1}+[2]_{p=5}+[8]_{5|q-1}+
\\\qquad\qquad+[2]_{8|q-1}-\left[\frac29\right]_{3|q-1}+\left[\frac59\right]_{3|q+1}+\left[\frac12\right]_{8|q-3}\end{array}$}\\\hline
\mbox{\small$3$}&\negmedspace\negmedspace\mbox{\small$\begin{array}{l}\frac{3641}{2880}q^5+
\frac{53}{144}q^4+\frac{83}{144}q^3-\frac89q^2+\frac{893}{960}q-\frac38+[\frac{67}{48}q^2-\frac43q-\frac{7}{16}]_{4|q-1}+\\\qquad\qquad+2[q-1]_{3|q-1}+
\frac19[5q+2]_{3|q+1}+[12]_{7|q-1}+[2]_{p=7}+[2]_{12|q-1}+\\
\qquad\qquad\qquad\qquad+\left[\frac12\right]_{8|q-1}+\left[\frac13\right]_{12|q-5}\end{array}$}\\\hline
\mbox{\small$4$}&\negmedspace\negmedspace\mbox{\small$\begin{array}{l}
\frac{28319}{22400}q^7+\frac{2119}{5760}q^6-\frac{2059}{9600}q^5+\frac{6143}{11520}q^4+
\frac{83}{1200}q^3+\frac{187}{5760}q^2-\frac{9}{1400}q-\frac{59}{1280}+\\\
+\left[-\frac{233}{384}q^3+\frac{99}{128}q^2-\frac{607}{384}q+\frac{117}{128}\right]_{4|q+1}
+4[q^2-q+1]_{3|q-1}+2[q^2-q]_{p=3}+\\\qquad+2[q-1]_{8|q-1}+\frac1{16}[7q+3]_{8|q+1}+
\frac{2}{25}[9q+4]_{5|q+1}+\frac{18}{25}[q-1]_{5|q-1}+\\\qquad\qquad+\left[\frac9{25}q-\frac15\right]_{p=5}+[12]_{9|q-1}+[4]_{16|q-1}+\left[\frac12\right]_{16|q-7}\end{array}$}\\\hline
\mbox{\small$5$}&\negmedspace\negmedspace\mbox{\small$\begin{array}{l}
\frac{27526069}{21772800}q^9+\frac{16481}{44800}q^8-\frac{778721}{3628800}q^7+\frac{11923}{86400}q^6+\frac{44881}{64800}q^5-\frac{43909}{43200}q^4+\frac{3133141}{3628800}q^3-\\\, -\frac{252227}{201600}q^2+\frac{357221}{161280}q-\frac{171}{256}+\left[\frac{5351}{3840}q^4-\frac{199}{160}q^3+\frac{521}{640}q^2-\frac{391}{240}q+\frac{597}{256}\right]_{4|q-1}+\\\  +\left[\frac{155}{324}q^2-\frac{167}{162}q+\frac{137}{972}\right]_{3|q+1}-\left[\frac{155}{324}q^2-\frac{241}{162}q+\frac{1361}{972}\right]_{3|q-1}+4[q-1]_{5|q-1}+\\
\ \  +\left[\frac{18}{25}q+\frac{8}{25}\right]_{5|q+1}+[20]_{11|q-1}+[2]_{p=11}+[4]_{20|q-1}+\left[\frac13\right]_{12|q-1}+\left[\frac25\right]_{20|q-9}\end{array}$}\\\hline
\end{tabular}
\end{table}
\end{center}

\begin{cor}\label{ohrat}
The dominant term of $\op{hyp}\rt(g)$ is
$$
\op{hyp}\rt(g)=
\left(1-\dfrac1{2!}+\dfrac1{3!}-\cdots-\dfrac1{(2g+2)!}\right)2q^{2g-1}+O(q^{2g-2}).
$$In particular, for $g$ large $\op{hyp}\rt(g)$ is asymptotically $(1-e^{-1})2q^{2g-1}$.
\end{cor}

\begin{proof}
The dominant term is the principal monomial of $2B(2g+2)$.
\end{proof}

\end{document}